\documentclass{article}

\usepackage{amsmath}
\usepackage{amsthm}
\usepackage{amsfonts}
\usepackage{amssymb}
\usepackage[all]{xy}
\usepackage{url}

\DeclareMathSymbol{\rightrightarrows}  {\mathrel}{AMSa}{"13}

\def\sd{\operatorname{sd}}

\def\Mor{\operatorname{Mor}}

\def\Pre{\operatorname{Pre}}

\def\Ex{\operatorname{Ex}}

\def\Pre{\mathbf{Pre}}

\catcode`\@=11
\def\varholim@#1#2{\mathop{\vtop{\ialign{##\crcr
 \hfil$#1\m@th\operator@font holim$\hfil\crcr
 \noalign{\nointerlineskip\kern\ex@}#2#1\crcr
 \noalign{\nointerlineskip\kern-\ex@}\crcr}}}}
\def\hocolim{\mathpalette\varholim@\rightarrowfill@} 
\def\hoinvlim{\mathpalette\varholim@\leftarrowfill@}
\catcode`\@=\active

\newtheorem{theorem}{Theorem}
\newtheorem{lemma}[theorem]{Lemma}
\newtheorem{corollary}[theorem]{Corollary}

\newtheorem{proposition}[theorem]{Proposition}

\theoremstyle{definition}

\newtheorem{remark}[theorem]{Remark}

\begin{document}

\title{Pro-equivalences of diagrams}
\author{J.F. Jardine\\
  Department of Mathematics\\
  University of Western Ontario}

\date{November 7, 2016}

\maketitle

\section*{Introduction}

The paper \cite{pro} constructs model structures for pro-objects in
simplicial presheaves. These constructions are based on the standard methods of
local homotopy theory \cite{LocHom}, and generalize known results for
pro-objects in ordinary categories of spaces \cite{EH},
\cite{Isaksen1}. The homotopy theory of pro-objects, as displayed in
all of these papers, is based on (and generalizes) standard features
of \'etale homotopy theory \cite{AM}, \cite{F}, and is an
artifact of the \'etale topology, and ultimately of Galois theory.

One can ask if there are intrinsic invariants which arise from
different Groth\-endieck topologies, which would generalize and extend
existing pro-homotopy theoretic techniques. The theory of cocycle
categories \cite{coc-cat}, \cite{LocHom} suggests this: it involves diagrams of weak
equivalences which extend classical diagrams of hypercovers, with
techniques that are present in all topologies. These diagrams have
very little structure, and must be considered in the context of
arbitrary small diagrams of simplicial presheaves, with a suitably
defined notions of ``pro-map'' and
``pro-equivalence''.

Suppose that $X: I \to s\Pre$ and $Y: J \to s\Pre$ are small diagrams
of simplicial presheaves.

If $I$ and $J$ are left filtered and $X$
and $Y$ are pro-objects, the most efficient way to define a
pro-map $\phi: X \to Y$ (following Grothendieck) is to say that $\phi$
is a natural transformation
\begin{equation*}
  \varinjlim_{j \in J}\ \hom(Y_{j},\ ) \to \varinjlim_{i \in I}\ \hom(X_{i},\ )
\end{equation*}
of the corresponding pro-representable functors. Then, as in
\cite{pro}, the map $\phi$ is a pro-equivalence if the induced map
\begin{equation*}
  \varinjlim_{j \in J}\ \mathbf{hom}(Y_{j},Z) \to
  \varinjlim_{i \in I}\ \mathbf{hom}(X_{i},Z)
\end{equation*}
of colimits of function spaces is a weak equivalence for all fibrant
objects $Z$. The displayed colimits are filtered, and filtered colimits are
already homotopy colimits, so this makes perfect sense from a homotopy
theoretic point of view.

In more general situations, in which the categories $I$ and $J$ are no
longer filtered, the colimits above must be replaced with homotopy
colimits. Then, for example, the homotopy colimit for the functor $j
\mapsto \hom(Y_{j},Z)$ is the nerve of the slice category $Y/Z$ whose
objects are the morphisms $Y_{j} \to Z$, and whose morphisms are
commutative diagrams
\begin{equation*}
  \xymatrix@R=6pt{
    Y_{j} \ar[dr] \ar[dd]_{\alpha_{\ast}} \\
& Z\\
    Y_{j'} \ar[ur]
  }
  \end{equation*}
in which $\alpha: j \to j'$ is a morphism of $J$. In this way, the
diagram $Y: J \to s\Pre$ homotopy pro-represents a functor $B(Y/?):
s\Pre \to s\mathbf{Set}$, with
\begin{equation*}
  Z \mapsto B(Y/Z).
\end{equation*}
A pro-map $X \to Y$ for arbitrary index categories is then naturally
defined as a natural transformation
\begin{equation*}
  B(Y/?) \to B(X/?)
\end{equation*}
of homotopy pro-representable functors. One can show that a pro-map,
so defined, consists of a functor $\alpha: J \to I$ and a natural
transformation $\theta: X\cdot \alpha \to Y$, and is therefore a
morphism of the Grothendieck construction associated to the list of
diagram categories $s\Pre^{I}$ and the functors between them which are
defined by restriction along functors $J \to I$.

Each such pro-map $(\alpha,\theta): X \to Y$ induces a commutative diagram
\begin{equation*}
  \xymatrix@C=8pt{
    \hocolim_{j \in J}\ \mathbf{hom}(Y_{j},Z) \ar[d] \ar[r]^-{\theta^{\ast}} &
    \hocolim_{j \in J}\ \mathbf{hom}(X_{\alpha(j)},Z) \ar[d] \ar[r]^-{\alpha_{\ast}} &
    \hocolim_{i \in I}\ \mathbf{hom}(X_{j},Z) \ar[d] \\
    BJ \ar[r]_{1} & BJ \ar[r]_{\alpha} & BI
  }
  \end{equation*}
and I say that the map $(\alpha,\theta)$ is a {\it pro-equivalence} if
the simplicial set map $\alpha: BJ \to BI$ is a weak equivalence, and
the top composite
\begin{equation*}
  \hocolim_{j \in J}\ \mathbf{hom}(Y_{j},Z) \to
  \hocolim_{i \in I}\ \mathbf{hom}(X_{j},Z)
\end{equation*}
is a weak equivalence of simplicial sets for each fibrant object $Z$.

The requirement that the map $\alpha: BJ \to BI$ be a weak equivalence
is not an issue if $I$ and $J$ are filtered categories, because the
spaces $BJ$ and $BI$ are contractible in that case. The definition of
pro-map between arbitrary small diagrams, while motivated by the
classical Grothendieck description as a transformation of
pro-representable functors, is  more rigid, because it involves a
comparison of functors that are represented by homotopy colimits.

If  $I$ is a fixed small category, then a natural
transformation $f: X \to Y$ of $I$-diagrams of simplicial presheaves
is a pro-map in the sense described above, and it is a pro-equivalence
if and only if it induces weak equivalences
\begin{equation*}
  f^{\ast}: \hocolim_{i \in I}\ \mathbf{hom}(Y_{i},Z) \to
  \hocolim_{i \in I}\ \mathbf{hom}(X_{i},Z)
\end{equation*}
for all injective fibrant objects $Z$. Observe that if $I$ is left
filtered, then the homotopy colimits can be replaced by filtered
colimits, and then the transformation $f$ would be a pro-equivalence
of pro-objects.
\medskip

One wants to show that the general description of
pro-equivalence that is displayed above is part of a homotopy
theoretic structure for small diagrams, but this has so far not been
realized.

The purpose of the present paper is more modest, to show that the
category of $I$-diagrams of simplicial presheaves, with ordinary
cofibrations and pro-equivalences, has a homotopy theoretic structure
in that it satisfies the axioms for a left proper closed simplicial
model category that is cofibrantly generated. This is the main result
of this paper, and appears as Theorem \ref{th xxx} below.

This has not previously been done for classical pro-objects, and that
theory is worked out as a test case in Section 1. The corresponding
model structure is given by Proposition \ref{prop 5}. 

The main steps in the arguments for Proposition \ref{prop 5} and
Theorem \ref{th xxx} are bounded monomorphism statements (in the style
of \cite{LocHom}), which appear as Lemma \ref{lem 1} and Lemma
\ref{lem 11}, respectively. Lemma \ref{lem 1}
is a special case of Lemma \ref{lem 11}, but it appears in the first section
with a much easier argument that makes heavy use of standard filtered
colimit techniques. 

The proof of Lemma \ref{lem 11} is more interesting, in that it
involves restrictions of $I$-diagrams to finite diagrams defined on
order complexes of finite simplicial complexes, in conjunction with
subdivision arguments that appeal to the stalklike structure of Kan's
$\Ex^{\infty}$ construction.

\tableofcontents

\section{The left filtered case}

Suppose that $I$ is a small category, and let $\mathcal{C}$ be a small
Grothendieck site.

In all of this paper, we assume that $\alpha$ is a regular cardinal
such that $\alpha > \vert \Mor(\mathcal{C}) \vert$ and $\alpha > \vert
I \vert$, and define the injective fibrant model functor $X \mapsto
L(X)$ for simplicial presheaves $X$ by formally inverting the
$\alpha$-bounded trivial cofibrations for the injective model
structure on the simplicial presheaf category $s\Pre$.

In this case, we know that there is a regular cardinal $\lambda >
\alpha$, such that if $X$ is a simplicial presheaf such that $\vert X
\vert < \lambda$ then the fibrant model $L(X)$ satisfies $\vert L(X)
\vert < \lambda$. Further, the fibrant model construction satisfies
\begin{equation*}
  L(X) = \varinjlim_{Y \in B_{\lambda}(X)}\ L(Y),
\end{equation*}
where $B_{\lambda}(X)$ denotes the family of $\lambda$-bounded
subobjects of $X$. The functor $X \mapsto L(X)$ also preserves
monomorphisms and intersections. See Lemma 7.16 of \cite{LocHom}.

It follows, for example, that,
given a general lifting problem of simplicial presheaves
\begin{equation}\label{eq 1}
\xymatrix{
A \ar[r]^{f} \ar[d]_{i} & Z \\
B \ar@{.>}[ur]
}
\end{equation}
 with $i$ a cofibration, $A$ $\lambda$-bounded and $Z$ injective fibrant, we can
 replace $Z$ by the $\lambda$-bounded object $L(A)$. In effect, we find a 
 factorization
\begin{equation*}
\xymatrix{
A \ar[r]^{j} \ar[dr]_{f} & L(A) \ar[d] \\
& Z
}
\end{equation*}
of the map $f$, where the map $j$ is the fibrant model. It follows that, if we
can solve the lifting problem
\begin{equation*}
\xymatrix{
A \ar[r]^{j} \ar[d]_{i} & L(A) \\
B \ar@{.>}[ur]
}
\end{equation*}
then we can solve the lifting problem (\ref{eq 1}) for all injective
fibrant objects $Z$.
\medskip

Suppose, for the rest of this section, that the index category $I$ is
left filtered.

Suppose given a
diagram of cofibrations 
\begin{equation*}
\xymatrix{
& X \ar[d]^{i} \\
A \ar[r] & Y
}
\end{equation*}
where $A$ is $\lambda$-bounded and $i$ is a pro-equivalence. 

The map 
\begin{equation*}
\varinjlim_{i}\ \mathbf{hom}(Y_{i},Z) \to \varinjlim_{i}\ \mathbf{hom}(X_{i},Z)
\end{equation*}
is a trivial fibration of simplicial sets for all injective fibrant
objects $Z$. This is equivalent to the assertion that any lifting
problem
\begin{equation}\label{eq 2}
\xymatrix{
(\partial\Delta^{n} \times Y_{i}) \cup (\Delta^{n} \times X_{i}) \ar[r]^-{f} \ar[d]
& Z \\
\Delta^{n} \times Y_{i} \ar@{.>}[ur] 
}
\end{equation}
can be solved after refining along $i$. Solution after refinement
means that there is a morphism $\alpha: j \to i$ of $I$ with a
commutative diagram
\begin{equation*}
\xymatrix{
(\partial\Delta^{n} \times Y_{j}) \cup (\Delta^{n} \times X_{j}) \ar[r]^{\alpha_{\ast}} \ar[d] & (\partial\Delta^{n} \times Y_{i}) \cup (\Delta^{n} \times X_{i})  \ar[r]^-{f} 
& Z \\
\Delta^{n} \times Y_{j} \ar@<-1ex>[urr]_{\theta}
}
\end{equation*}
We replace the map to $Z$
in the picture by the fibrant model
\begin{equation*}
(\partial\Delta^{n} \times Y_{i}) \cup (\Delta^{n} \times X_{i}) \xrightarrow{j} Z(Y)_{i}.
\end{equation*}

Write $Z(B)_{i}$ for the fibrant model of the object
\begin{equation*}
  (\partial\Delta^{n} \times B_{i}) \cup (\Delta^{n} \times (B_{i} \cap
  X_{i})),
\end{equation*}
where $B$ varies through the $\lambda$-bounded subobjects of $Y$.
There is a relation
\begin{equation*}
\varinjlim_{B}\ Z(B)_{i} = Z(Y)_{i}
\end{equation*}
where the colimit is indexed over the $\lambda$-bounded subobjects $B$
of $Y$. In effect, every $\lambda$-bounded subobject $C$ of
\begin{equation*}
  (\partial\Delta^{n} \times Y_{i}) \cup (\Delta^{n} \times X_{i})
\end{equation*}
is contained in some 
\begin{equation*}
  (\partial\Delta^{n} \times B_{i}) \cup (\Delta^{n} \times (B_{i} \cap
  X_{i})),
\end{equation*}
with $B \subset Y$ $\lambda$-bounded, while $Z(Y_{i})$ is a filtered colimit of the objects $Z(C)$.

It follows that the image of the composite map
\begin{equation*}
  \Delta^{n} \times A_{j} \to \Delta^{n} \times Y_{j} \xrightarrow{\theta} Z(Y)_{i}
  \end{equation*}
lies in $Z(B)_{i}$ for some $\lambda$-bounded $B$ such that $A
\subset B$.

This is the start of an inductive process. There is a
$\lambda$-bounded subobject $B$ with $A \subset B$ such that every lifting
problem
\begin{equation*}
\xymatrix{
(\partial\Delta^{n} \times A_{i}) \cup (\Delta^{n} \times (A_{i} \cap X_{i}))
\ar[r]^-{j} \ar[d] & Z(A)_{i} \\
\Delta^{n} \times A_{i} \ar@{.>}[ur]
}
\end{equation*}
(i.e. for all $i$) is solved over $B$ after refinement. It follows
that there is a $\lambda$-bounded $C$ with $A \subset C \subset Y$
such that every lifting problem
\begin{equation*}
\xymatrix{
(\partial\Delta^{n} \times C_{i}) \cup (\Delta^{n} \times (C_{i} \cap X_{i}))
\ar[r]^-{j} \ar[d] & Z(C)_{i} \\
\Delta^{n} \times C_{i} \ar@{.>}[ur]
}
\end{equation*}
is solved after refinement, and this for all $i$.

We have proved the following bounded
monomorphism statement:

\begin{lemma}\label{lem 1} 
Suppose that $I$ is a left filtered category. Suppose given a
diagram of cofibrations 
\begin{equation*}
\xymatrix{
& X \ar[d]^{i} \\
A \ar[r] & Y
}
\end{equation*}
where $A$ is $\lambda$-bounded and $i$ is a pro-equivalence.  Then
there is an $\lambda$-bounded subobject $C \subset Y$ with $A \subset
C$ such that the map $C \cap X \to C$ is a pro-equivalence.
\end{lemma}

Say that a map $p: X \to Y$ is a {\it pro-fibration} if it has the right
lifting property with respect to all cofibrations $A \to B$ which are
pro-equivalences.

\begin{lemma}\label{lem 2}
A map $p: X \to Y$ is a pro-fibration and a pro-equivalence if and
only if it has the right lifting property with respect to all
cofibrations.
\end{lemma}

\begin{proof}
Suppose that $p$ is a pro-fibration and a pro-equivalence. We show
that it has the right lifting property with respect to all
$\lambda$-bounded cofibrations.

The converse assertion is clear: if
$p$ has the right lifting property with respect to all cofibrations,
then it is a sectionwise equivalence and hence a pro-equivalence, and
it has the right lifting property with respect to a cofibrations which
are pro-equivalences.
\medskip

Suppose given a lifting problem
\begin{equation*}
\xymatrix{
A \ar[r] \ar[d]_{i} & X \ar[d]^{p} \\
B \ar[r] \ar@{.>}[ur] & Y
}
\end{equation*}
where $i$ is an $\lambda$-bounded cofibration and $p$ is a
pro-fibration and a pro-equivalence. We show that the indcated lift
exists.

This will be true for all $\lambda$-bounded cofibrations, and these
generate the class of all cofibrations, so it follows that $p$ has
the right lifting property with respect to all cofibrations.

Factorize $p$ as $p=q
\cdot j$ where $q$ is a trivial injective fibration and $j$ is a
cofibration.  Then there is a diagram
\begin{equation*}
\xymatrix@R=10pt{
A \ar[rr] \ar[dd]_{i} && X \ar[dd]^{p} \ar[dl]_{j} \\
& V \ar[dr]^{q} \\
B \ar[ur]^{\theta} \ar[rr] && Y
}
\end{equation*}
The cofibration $j$ is a pro-equivalence, and the image $\theta(B)$ is
$\lambda$-bounded, so there is a subobject $D \subset V$ with
$\theta(B) \subset D$, such that the map $D \cap X \to D$ is a
pro-equivalence, by Lemma \ref{lem 1}. We have found a factorization
\begin{equation*}
\xymatrix{
A \ar[r] \ar[d] & D \cap X \ar[d] \ar[r] & X \ar[d] \\
B \ar[r] & D \ar[r] \ar@{.>}[ur] & Y
}
\end{equation*}
of the original diagram, with a pro-equivalence $D \cap X \to D$, and
the lifting problem has the indicated solution.
\end{proof}

\begin{corollary}\label{cor 3}
A map $p: X \to Y$ is a pro-fibration and a pro-equivalence if and
only if it is a trivial injective fibration.
\end{corollary}

The following statements are also clear:

\begin{lemma}\label{lem 4}
\begin{itemize}
\item[1)] A map $p$ is a pro-fibration if and only it has the right
  lifting property with respect to all $\lambda$-bounded cofibrations
  which are pro-equivalences.
\item[2)] The class of maps which are cofibrations and
  pro-equivalences is closed under pushout.
\end{itemize}
\end{lemma}

We can now prove the following:

\begin{proposition}\label{prop 5}
Suppose that the category $I$ is left filtered, and that $\mathcal{C}$
is a Grothendieck site. The category $s\Pre^{I}$ of $I$-diagrams in
simplicial presheaves on $\mathcal{C}$, together with the classes of
cofibrations, pro-equivalences and pro-fibrations, satisfies the
axioms for a left proper closed simplicial model category. This model
structure is cofibrantly generated.
\end{proposition}

\begin{proof}
The factorization axiom {\bf CM5} follows from Corollary \ref{cor 3}
and Lemma \ref{lem 4}. The lifting axiom {\bf CM4} follows from
Corollary \ref{cor 3}. The other closed model axioms are
automatically true. The function complex $\mathbf{hom}(X,Y)$ is the
standard one, and one shows that, given a cofibration $A \to B$ of
$I$-diagrams and a cofibration $K \to L$ of simplicial sets, then the
map
\begin{equation*}
(A \times L) \cup (B \times K) \to B \times L
\end{equation*}
is a cofibration which is a pro-equivalence if $A \to B$ is a
pro-equivalence or $K \to L$ is a weak equivalence.

Left properness is trivial to verify, and cofibrant generation is a
consequence of Lemma \ref{lem 1}.
\end{proof}

\section{Subdivisions and simplex categories}

Suppose that $X$ is a simplicial set. The poset $NX$ has objects given
by all non-degenerate simplices $\sigma \in X$, and we say that
$\sigma \leq \tau$ if $\sigma$ is a member of the subcomplex of $X$
which is generated by $\tau$.

Suppose that $K$ is a finite simplicial complex, in the sense that $K$
is a subcomplex of some simplex $\Delta^{N}$. Then $\sigma \leq \tau$
in $NK$ if $\sigma$ is a face of $\tau$; furthermore, $\sigma$ is a
face of $\tau$ in a unique way. The nerve $BNK$ is often said to be
the {\it order complex} of the simplicial complex $K$ \cite{Kozlov}.
The non-degenerate
simplices $\sigma$ of $K$ are those for which the canonical map
$\sigma: \Delta^{n} \to K$ is a monomorphism.

The {\it simplex category} $\mathbf{\Delta}/X$ for a simplicial set
$X$ has objects consisting of all simplicial set maps $\sigma:
\Delta^{n} \to X$. The morphisms $\theta: \tau \to \sigma$ are
commutative diagrams of simplicial set maps
    \begin{equation*}
      \xymatrix@R=8pt{
        \Delta^{m} \ar[dr]^{\tau} \ar[dd]_{\theta} \\
        & X \\
        \Delta^{n} \ar[ur]_{\sigma}
      }
      \end{equation*}
    
Suppose that $X$ is a simplicial set and that $L$ is a finite
simplicial complex. A map $f: L \to X$ determines a functor $f:
\mathbf{\Delta}/L \to \mathbf{\Delta}/X$ of simplex categories in the
obvious way.

Since $L$ is a simplicial complex, there is an inclusion functor $NL
\to \mathbf{\Delta}/L$, and we have the following:

\begin{lemma}\label{lem 6}
  Suppose that $L$ is a finite simplicial complex.
  Then the composite map
  \begin{equation*}
    \varinjlim_{\Delta^{n} \subset L} \Delta^{n} \to \varinjlim_{\Delta^{n} \to L}\ \Delta^{n} \xrightarrow{\cong} L
  \end{equation*}
  is an isomorphism. A simplicial set map
$f: L \to X$ is
defined by the composite functor
\begin{equation*}
\tilde{f}: NL \to \mathbf{\Delta}/L \xrightarrow{f_{\ast}} \mathbf{\Delta}/X.
\end{equation*}
\end{lemma}

 Lemma \ref{lem 6} says that a finite simplicial complex is a colimit
 of its non-degenerate simplices. 

\begin{proof}
The $r$-simplices $x \in \Delta^{n}
\xrightarrow{\sigma} L$ and $y \in \Delta^{m} \xrightarrow{\tau} L$
have the same image in $K$ if and only if there is a diagram
\begin{equation*}
\xymatrix{
\Delta^{r} \ar[r]^{x} \ar[d]_{y} & \Delta^{n} \ar[d]^{\sigma} \\
\Delta^{m} \ar[r]_{\tau} & L
}
\end{equation*}
Since $L$ is a finite simplicial complex, the pullback
$\Delta^{m} \times_{L} \Delta^{n}$ is a non-degenerate simplex of $L$
if $\sigma$ and $\tau$ are non-degenerate. It follows that the composite map
\begin{equation*}
\varinjlim_{\Delta^{n} \subset L}\ \Delta^{n} \to \varinjlim_{\Delta^{m} \to L}\ \Delta^{m} \xrightarrow{\cong} L
\end{equation*}
is an isomorphism, and so there is a commutative diagram
\begin{equation*}
  \xymatrix{
    \varinjlim_{\Delta^{n} \subset L}\ \Delta^{n} \ar[r]^{\tilde{f}_{\ast}} \ar[d]_{\cong}
    & \varinjlim_{\Delta^{n} \to X}\ \Delta^{n} \ar[d]^{\cong} \\
    L \ar[r]_{f} & X
  }
  \end{equation*}
\end{proof}

\begin{corollary}\label{cor 7}
Suppose that $g: K \to L$ is a morphism of simplicial complexes.  Then
the map $g$ is induced by a functor $g_{\ast}: NK \to NL$ which takes a
non-degenerate simplex $\sigma$ to the non-degenerate simplex
which generates the subcomplex $\langle g(\sigma) \rangle$ of $L$.
\end{corollary}

\begin{proof}
  Consider the picture
  \begin{equation*}
    \xymatrix{
      \varinjlim_{\Delta^{n} \subset K}\ \Delta^{n} \ar[d]_{\cong} \ar@{.>}[r]
      & \varinjlim_{\Delta^{m} \subset L}\ \Delta^{m} \ar[d]^{\cong} \\
      \varinjlim_{\Delta^{n} \to K}\ \Delta^{n} \ar[d]_{\cong} \ar[r]^{g_{\ast}}
      & \varinjlim_{\Delta^{m} \to L}\ \Delta^{m} \ar[d]^{\cong} \\
      K \ar[r]_{g} & L
      }
  \end{equation*}
  in which the vertical maps are isomorphisms by Lemma \ref{lem 6}.

  If $\sigma: \Delta^{n} \to K$ is a non-degenerate simplex, then the composite
  \begin{equation*}
    \Delta^{n} \xrightarrow{\sigma} K \xrightarrow{g} L
  \end{equation*}
  has a unique factorization
  \begin{equation*}
    \Delta^{n} \xrightarrow{s_{\sigma}} \Delta^{r} \xrightarrow{d_{\sigma}} L,
  \end{equation*}
  where $s_{\sigma}$ is a codegeneracy and $d_{\sigma}$ is a
  non-degenerate simplex of $L$. Write $g_{\ast}(\sigma) =
  d_{\sigma}$, and observe that the assignment $\sigma \mapsto
  g_{\ast}(\sigma)$ defines a functor $g_{\ast}: NK \to NL$. The
  codegeneracies $s_{\sigma}$ and the functor $g_{\ast}$ define the
  dotted arrow map in the diagram, in the sense that there are commutative diagrams
  \begin{equation*}
    \xymatrix{
      \Delta^{n} \ar[r]^{s_{\sigma}} \ar[d]_{in_{\sigma}}
      & \Delta^{r} \ar[d]^{in_{g_{\ast}(\sigma)}} \\
        \varinjlim_{\Delta^{n} \subset K}\ \Delta^{n} \ar@{.>}[r]
        & \varinjlim_{\Delta^{m} \subset L}\ \Delta^{m}
    }
    \end{equation*}
\end{proof}

\begin{remark}
A similar argument shows that the inclusion $NL \subset
\mathbf{\Delta}/L$ is a homotopy equivalence of categories for each finite
simplicial complex $L$. Every simplex $\sigma:\Delta^{n} \to L$ has a
canonical factorization $\sigma = d_{\sigma}s_{\sigma}$, and the
assignment $\sigma \mapsto d_{\sigma}$ defines a functor
$\mathbf{\Delta}/L \to NL$ which is inverse to the inclusion up to
homotopy defined by the morphisms $s_{\sigma}$.
  \end{remark}

\begin{corollary}\label{cor 9}
  Suppose given maps
  \begin{equation*}
    K \xrightarrow{g} L \xrightarrow{f} X
    \end{equation*}
where $K$ and $L$ are finite simplicial complexes. Then the composite
$f \cdot g$ is induced by the composite functor
\begin{equation*}
  NK \xrightarrow{g_{\ast}} NL \xrightarrow{\tilde{f}} \mathbf{\Delta}/X.
\end{equation*}
  \end{corollary}

The {\it subdivision} $\sd(\Delta^{n})$ is defined to be the nerve of the
corresponding order complex:
\begin{equation*}
  \sd(\Delta^{n}) = BN\Delta^{n}.
\end{equation*}
More generally, the {\it subdivision} $\sd(X)$ of a simplicial set $X$
is defined by the assignment
\begin{equation*}
  \sd(X) = \varinjlim_{\Delta^{m} \to X}\ \sd(\Delta^{m}).
  \end{equation*}

There is a map
\begin{equation*}
  \pi: \sd(X) \to BNX
\end{equation*}
that is natural in simplicial sets $X$, and is defined by the composites
\begin{equation*}
  \sd(\Delta^{n}) \xrightarrow{\cong} BN\Delta^{n} \xrightarrow{\sigma_{\ast}} BNX.
\end{equation*}
arising from the simplices $\Delta^{n} \to X$ of $X$.

\begin{lemma}\label{lem 10}
The map $\pi: \sd(K) \to BNK$ is an
isomorphism if $K$ is a finite simplicial complex.
\end{lemma}

\begin{proof}
  There is a commutative diagram
  \begin{equation*}
    \xymatrix{
      \varinjlim_{ \sigma \in NK}\ \sd \langle \sigma \rangle
      \ar[r]^-{\cong} \ar[d]_{\pi_{\ast}}
      & \sd(K) \ar[d]^{\pi} \\
      \varinjlim_{\sigma\in NK}\ BN\langle \sigma \rangle \ar[r]_-{\cong} & BNK
    }
    \end{equation*}
where $\langle \sigma \rangle$ is the subcomplex of $X$ that is
generated by the simplex $\sigma$. The top horizontal map is an
isomorphism, since $K$ is a colimit of the subcomplexes $\langle
\sigma \rangle$ which are generated by non-degenerate simplices
$\sigma$ by Lemma \ref{lem 6}. The bottom horizontal map is an
isomorphism, again since the intersection of two non-degenerate simplices
of the simplicial complex $K$ is a non-degenerate simplex.

All maps $\pi: \sd \langle \sigma \rangle \to BN\langle \sigma
\rangle$ are isomorphisms, again since $K$ is a simplicial
complex: if $\sigma$ is a non-degenerate $n$-simplex
of $K$, then the canonical map $\Delta^{n} \to \langle \sigma \rangle$
that is defined by $\sigma$ is an isomorphism.
\end{proof}

Suppose again that $K \subset \Delta^{N}$ is a finite simplicial
complex. The induced functor $NK \to N\Delta^{N}$ is a fully faithful
imbedding, which induces a monomorphism $\sd(K) \to \sd(\Delta^{N})$
of associated nerves. There is a total ordering on the non-degenerate
simplices of $\Delta^{N}$ which extends the total ordering on its
vertices, such that every $k$-simplex is less than every
$(k+1)$-simplex for $0 \leq k \leq N-1$. In this way, there is a fully
faithful imbedding $N\Delta^{N} \subset \mathbf{M}$ for some ordinal
number $M$, and so there is a monomorphism $\sd(\Delta^{N}) \subset
\Delta^{M}$. It follows from the resulting string of inclusions
\begin{equation*}
  \sd(K) \subset \sd(\Delta^{N}) \subset \Delta^{M}
\end{equation*}
that $\sd(K)$ is a finite simplicial complex.

The {\it last vertex} functor $N\Delta^{n} \to \mathbf{n}$ is defined
by sending a non-degenerate simplex $\sigma: \mathbf{k} \to
\mathbf{n}$ to $\sigma(k)$ --- see \cite{GJ},
\cite{J34}. This functor induces a simplicial set map $\gamma:
\sd\Delta^{n} \to \Delta^{n}$. The maps $\gamma$ are natural in
simplices $\Delta^{n}$, and together induce a map
\begin{equation*}
 \gamma: \sd(X) \to X
\end{equation*}
which is natural in simplicial sets $X$. Composition of instances of this map defines the various natural maps
\begin{equation*}
  \sd^{k}(X) \xrightarrow{\gamma} \sd^{k-1}(X) \xrightarrow{\gamma} \dots \xrightarrow{\gamma} \sd(X) \xrightarrow{\gamma} X,
\end{equation*}
all of which will be denoted by $\gamma$, and called {\it subdivision maps}.

It is a consequence of the results of this section (specifically,
Corollary \ref{cor 9}) that if $K$ is a finite simplicial complex and
$X$ is a simplicial set, then the string of simplicial set maps
\begin{equation*}
  \sd^{k}(K) \xrightarrow{\gamma} \sd^{k-1}(K) \xrightarrow{\gamma} \dots \xrightarrow{\gamma} \sd(K) \xrightarrow{\gamma} K \xrightarrow{f} X
\end{equation*}
is induced by the string of functors
\begin{equation*}
  N\sd^{k}(K) \xrightarrow{\gamma_{\ast}} N\sd^{k-1}(K) \xrightarrow{\gamma_{\ast}} \dots \xrightarrow{\gamma_{\ast}} N\sd(K) \xrightarrow{\gamma_{\ast}} NK \xrightarrow{\tilde{f}} \mathbf{\Delta}/X.
\end{equation*}

Suppose that $f: X \to Y$ is a map of Kan complexes, and replace $f$
by a fibration in the usual way, by forming the pullback diagram
\begin{equation*}
\xymatrix{
X \times_{Y} Y^{I} \ar[r]^{f_{\ast}} \ar[d] & Y^{I} \ar[r]^{d_{1}} \ar[d]^{d_{0}} & Y \\
X \ar[r]_{f} & Y
}
\end{equation*}
Let $\pi$ be the composite $d_{1} \cdot f_{\ast}$. Then $\pi$ is a
fibration which is weakly equivalent to $f$.

Here, $Y^{I}$ is the function complex $\mathbf{hom}(\Delta^{1},Y)$,
and the maps $d_{0},d_{1}: Y^{I} \to Y$ are defined by precomposition
with the maps $d^{0},d^{1}: \Delta^{0} \to \Delta^{1}$, respectively.

A solution of the lifting problem
\begin{equation*}
\xymatrix{
\partial\Delta^{n} \ar[r]^{(\alpha,h_{\ast})} \ar[d] & X \times_{Y} Y^{I} \ar[d]^{\pi} \\
\Delta^{n} \ar[r]_{\beta} \ar@{.>}[ur] & Y
}
\end{equation*}
is equivalent an extension of the adjoint diagram
\begin{equation}\label{eq 3}
\xymatrix{
\partial\Delta^{n} \ar[r]^{\alpha} \ar[d]_{d_{0}} & X \ar[d]^{f} \\
(\partial\Delta^{n} \times \Delta^{1}) \cup (\Delta^{n} \times \{ 0 \}) \ar[r]_-{(h,\beta)} & Y
}
\end{equation}
to a diagram
\begin{equation}\label{eq 4}
  \xymatrix{
    \Delta^{n} \ar[r]^{\theta} \ar[d]_{d_{0}} & X \ar[d]^{f} \\
    \Delta^{n} \times \Delta^{1} \ar[r]_-{H} & Y
  }
\end{equation}

It follows that a simplicial set map $f: X \to Y$ between Kan
complexes is a weak equivalence if and only if every diagram
(\ref{eq 3}) extends to a diagram (\ref{eq 4}).

If $X$ and $Y$ are not Kan complexes, then $f: X \to Y$ is a weak
equivalence if and only if all diagrams (\ref{eq 3}) extend to
diagrams (\ref{eq 4}) after subdivision. In other words, given a
diagram (\ref{eq 3}), there is some $k \geq 0$ such that the composite
diagram
\begin{equation*}
\xymatrix{
\sd^{k}(\partial\Delta^{n}) \ar[r]^{\gamma} \ar[d]_{d_{0\ast}} & \partial\Delta^{n} \ar[r]^{\alpha}  & X \ar[d]^{f} \\
\sd^{k}((\partial\Delta^{n} \times \Delta^{1}) \cup (\Delta^{n} \times \{ 0 \})) \ar[r]_-{\gamma}
& (\partial\Delta^{n} \times \Delta^{1}) \cup (\Delta^{n} \times \{ 0 \}) \ar[r]_-{(h,\beta)} & Y
}
\end{equation*}
extends to a diagram
\begin{equation*}
  \xymatrix{
\sd^{k}(\Delta^{n}) \ar[r]^{\gamma} \ar[d]_{d_{0\ast}} & \Delta^{n} \ar[r]^{\theta} & X \ar[d]^{f} \\
\sd^{k}(\Delta^{n} \times \Delta^{1}) \ar[r]_-{\gamma} &    \Delta^{n} \times \Delta^{1} \ar[r]_-{H} & Y
  }
\end{equation*}
This follows from the fact that the map $f: X \to Y$ is a weak
equivalence if and only if the induced map $\Ex^{\infty}X \to
\Ex^{\infty}Y$ is a weak equivalence of Kan complexes --- see section
III.4 of \cite{GJ}.

\section{General diagram categories}

Suppose that $I$ is a fixed (but arbitrary) small index category.

The definitions of Section 1 persist for the category $s\Pre^{I}$ of
$I$-diagrams of simplicial presheaves with their natural
transformations. A {\it cofibration} is a sectionwise monomorphism. A
{\it pro-equivalence} $X \to Y$ is a natural transformation which
induces a weak equivalence of simplicial sets
\begin{equation*}
 f^{\ast}:  \hocolim_{i \in I}\ \mathbf{hom}(Y_{i},Z) \to \hocolim_{i \in I}\ \mathbf{hom}(X_{i},Z)
\end{equation*}
for all injective fibrant objects $Z$..Finally, a {\it pro-fibration} is a
map which has the right lifting property with respect to maps which
are cofibrations and pro-equivalences.

We prove, in this section, an analogue of Proposition \ref{prop 5},
which asserts that these definitions give the category $s\Pre^{I}$ the
structure of a left proper closed simplicial model category. The main
result is Theorem \ref{th xxx} below.
\medskip

Suppose that $X$ is an $I$-diagram of simplicial presheaves
and that the simplicial presheaf $Z$ is injective fibrant. Suppose
that $K$ is a finite simplicial complex.

A simplicial set map $f: K \to \hocolim_{I^{op}}\ \mathbf{hom}(X,Z)$ is induced by
a functor
\begin{equation*}
\tilde{f}: NK \to \mathbf{\Delta}/\hocolim_{I^{op}}\ \mathbf{hom}(X,Z),
\end{equation*}
according to Lemma \ref{lem 6}.
The homotopy colimit can be identified with the diagonal of the nerve
of a simplicial category $H_{I}(X,Z)$ with morphisms in simplicial
degree $n$ having the form
\begin{equation*}
  \xymatrix@C=10pt{
    X(i) \times \Delta^{n} \ar[rr]^{\alpha_{\ast} \times 1} \ar[dr] && X(j) \times \Delta^{n} \ar[dl] \\
    & Z
  }
\end{equation*}
where $\alpha: i \to j$ is a morphism of $I$.
Observe that there is a
forgetful functor
\begin{equation*}
\pi: H_{I}(X,Z) \to I. 
\end{equation*}

An $n$-simplex $D$ of
the homotopy colimit consists of a functor $\alpha:
\mathbf{n} \to I$ and a diagram
\begin{equation*}
  \xymatrix@C=8pt{
    X(\alpha(0)) \times \Delta^{n} \ar[drr] \ar[r]
    & X(\alpha(1)) \times \Delta^{n} \ar[dr] \ar[r] 
    & \dots \ar[r] & X(\alpha(n-1)) \times \Delta^{n} \ar[dl] \ar[r]
    & X(\alpha(n)) \times \Delta^{n} \ar[dll]^{\tau} \\
    & & Z
  }
\end{equation*}
or alternatively an $n$-simplex in the nerve of the slice category $X
\times \Delta^{n}/Z$. Note (this is standard) that the simplex is
completely determined by the functor $\alpha$ and the map $\tau$.

Given such an object, if $\theta: \mathbf{m} \to \mathbf{n}$ is an
ordinal number map then the simplex $\theta^{\ast}D$ is defined by the
composite functor $\mathbf{m} \xrightarrow{\theta} \mathbf{n}
\xrightarrow{\alpha} I$ and the composite diagram
\begin{equation*}
  \xymatrix@C=6pt{
    X(\alpha(\theta(0))) \times \Delta^{m} \ar[d]_{1 \times \theta} \ar[r]
    & X(\alpha(\theta(1))) \times \Delta^{m} \ar[d]^{1 \times \theta} \ar[r] 
    & \dots \ar[r]
    & X(\alpha(\theta(m-1))) \times \Delta^{m} \ar[d]^{1 \times \theta} \ar[r]
    & X(\alpha(\theta(m))) \times \Delta^{m} \ar[d]^{1 \times \theta} \\
    X(\alpha(\theta(0))) \times \Delta^{n} \ar[drr] \ar[r]
    & X(\alpha(\theta(1))) \times \Delta^{n} \ar[dr] \ar[r] 
    & \dots \ar[r] & X(\alpha(\theta(m-1))) \times \Delta^{n} \ar[dl] \ar[r]
    & X(\alpha(\theta(m))) \times \Delta^{n} \ar[dll] \\
    & & Z
  }
\end{equation*}

Alternatively, the $n$-simplex above is a functor
$\alpha: \mathbf{n} \to I$ and a natural transformation $f: (X\cdot
\alpha) \times \Delta^{n} \to Z$. Then $\theta^{\ast}(\alpha,f)$ is
the pair consisting of the composite functor $\alpha \cdot \theta$
  together with the composite natural transformation
\begin{equation*}
  (X\cdot\alpha\cdot\theta) \times \Delta^{m} \xrightarrow{1 \times \theta}
  (X\cdot\alpha\cdot\theta) \times \Delta^{n} \xrightarrow{f\cdot \theta} Z.
\end{equation*}

Write
\begin{equation*}
E_{I}(X,Z) =
\mathbf{\Delta}/\hocolim_{I^{op}}\ \mathbf{hom}(X,Z)
\end{equation*}
to make the
notation easier to deal with.
\medskip

Suppose that $i: X \to Y$ is a monomorphism of $I$-diagrams. Let $j: K
\subset L$ be an inclusion of finite simplicial complexes, and
consider a diagram
\begin{equation*}
  \xymatrix{
    K \ar[r] \ar[d] & \hocolim_{I^{op}}\ \mathbf{hom}(Y,Z) \ar[d] \\
    L \ar[r] & \hocolim_{I^{op}}\ \mathbf{hom}(X,Z)
  }
\end{equation*}
Converting to functors by using the methods of the last section
(Corollary \ref{cor 9}) gives the diagram of functors
\begin{equation*}
  \xymatrix{
    NK \ar[r]^-{\omega} \ar[d]_{j} & E_{I}(Y,Z) \ar[d]^{i^{\ast}} \\
    NL \ar[r]_-{\beta} & E_{I}(X,Z)
  }
\end{equation*}
The functor $i^{\ast}$ is defined by restriction to $X$.
The diagram consists of a functor $\omega$ whose
restriction to $X$ extends to a functor $\beta$ that is defined on
$NL$.

There is a functor 
\begin{equation*}
v_{Y}: E_{I}(Y,Z) \to s\mathbf{Pre},
\end{equation*}
which takes any $n$-simplex $f: Y\cdot \alpha \times \Delta^{n} \to Z$
to the simplicial set $Y(\alpha(n)) \times \Delta^{n}$ (which is the
colimit of $Y \cdot \alpha \times \Delta^{n}$).  This functor $v_{Y}$
is independent of $Z$: if $Z \to W$ is a simplicial set map, then the
diagram
\begin{equation*}
\xymatrix@C=10pt{
E_{I}(Y,Z) \ar[rr] \ar[dr]_{v_{Y}} && E_{I}(Y,W) \ar[dl]^{v_{Y}} \\
& s\mathbf{Pre}
}
\end{equation*}
commutes. 

The functor $\omega: NK \to E_{I}(Y,Z)$ can be identified
with a natural transformation
\begin{equation*}
v_{Y}\cdot \omega \to Z.
\end{equation*}
Taking the colimit $L_{K}(Y)$ of the
composite functor
\begin{equation*}
  NK \xrightarrow{\omega} E_{I}(Y,Z) \xrightarrow{v_{Y}} s\mathbf{Pre}
\end{equation*}
therefore defines a functor $\omega_{\ast}: NK \to E_{I}(Y,L_{K}(Y))$
and a simplicial set map $f_{\omega}: L_{K}(Y) \to Z$ such that
the diagram of functors
\begin{equation*}
  \xymatrix{
    NK \ar[r]^-{\omega_{\ast}} \ar[dr]_{\omega}
    & E_{I}(Y,L_{K}(Y)) \ar[d]^{f_{\omega\ast}} \\
    & E_{I}(Y,Z)
  }
\end{equation*}
commutes. 

Write $j: L_{K}(Y) \to \mathcal{L}(L_{K}(Y))$ for the natural
injective fibrant model of the simplicial presheaf $L_{K}(Y)$. The map
$f_{\gamma}: L_{K}(Y) \to Z$ factors through a map
$\mathcal{L}(L_{K}(Y)) \to Z$.
\medskip

Suppose given a commutative diagram of inclusions
\begin{equation}\label{eq 5}
\xymatrix{
K \ar[r] \ar[d] & K' \ar[d] \\
L \ar[r] & L'
}
\end{equation}
of finite complexes, and suppose that $i: X \to Y$ is a cofibration of
$I$-diagrams such that all diagrams
\begin{equation}\label{eq 6}
\xymatrix{
NK \ar[r]^-{\omega} \ar[d] & E_{I}(Y,Z) \ar[d] \\
NL \ar[r]_-{\beta} & E_{I}(X,Z)
}
\end{equation}
extend to diagrams
\begin{equation}\label{eq 7}
\xymatrix{
NK' \ar[r] \ar[d] & E_{I}(Y,Z) \ar[d] \\
NL' \ar[r] & E_{I}(X,Z)
}
\end{equation}
after subdivision, for all injective fibrant $Z$.

This means that, for each diagram (\ref{eq 6}) there is a $k \geq 0$
such that the diagram
\begin{equation*}
\xymatrix{
N\sd^{k}(K) \ar[r]^-{\gamma_{\ast}} \ar[d]  & NK \ar[r]^-{\omega} & E_{I}(Y,Z) \ar[d] \\
N\sd^{k}(L) \ar[r]_-{\gamma_{\ast}} & NL \ar[r]_-{\beta} & E_{I}(X,Z)
}
\end{equation*}
extends to a diagram
\begin{equation*}
\xymatrix{
N\sd^{k}(K') \ar[r] \ar[d] & E_{I}(Y,Z) \ar[d] \\
N\sd^{k}(L') \ar[r] & E_{I}(X,Z)
}
\end{equation*}

A commutative diagram (\ref{eq 6}) is, equivalently, a diagram of
simplicial presheaf maps
\begin{equation}\label{eq 8}
\xymatrix{
L_{K}X \ar[r] \ar[d] & L_{L}X \ar[d] \\
L_{K}Y \ar[r] & Z
}
\end{equation}
and to say that diagram (\ref{eq 6}) extends to a diagram (\ref{eq 7}) amounts to the assertion that there is a diagram
\begin{equation*}
\xymatrix{
L_{K'}X \ar[r] \ar[d] & L_{L'}X \ar[d] \\
L_{K'}Y \ar[r] & Z
}
\end{equation*}
which restricts to the given diagram (\ref{eq 8}) along the maps
\begin{equation*}
\xymatrix{
L_{K}Y \ar[d] & L_{K}X \ar[l] \ar[r] \ar[d] & L_{L}X \ar[d] \\
L_{K'}Y & L_{K'}X \ar[l] \ar[r] & L_{L'}X
}
\end{equation*}

In other words, we require the existence of a lifting in the diagram
\begin{equation}\label{eq 9}
\xymatrix{
L_{K}Y \cup_{L_{K}X} L_{L}X \ar[r] \ar[d] & Z \\
L_{K'}Y \cup_{L_{K'}X} L_{L'}X \ar@{.>}[ur]
}
\end{equation}
for all injective fibrant $Z$.

The requirement that (\ref{eq 6}) extends to (\ref{eq 7}) up to
subdivision amounts to the existence of a number $k \geq 0$ such that
the lift exists in the diagram
\begin{equation*}
\xymatrix{
  L_{\sd^{k}(K)}Y \cup_{L_{\sd^{k}K}(X)} L_{\sd^{k}(L)}X \ar[r] \ar[d]
  & L_{K}Y \cup_{L_{K}X} L_{L}X \ar[r] & Z \\
L_{\sd^{k}(K')}Y \cup_{L_{\sd^{k}(K')}X} L_{\sd^{k}(L')}X \ar@{.>}@<-1ex>[urr]
}
\end{equation*}

It is enough to solve the extension problem in the case where $Z$ is
an injective fibrant model of the pushout $L_{K}Y \cup_{L_{K}X}
L_{L}X$, since all extension problems (\ref{eq 9}) are solved by the
existence of extensions for this fibrant model.

We return to the cases of (\ref{eq 3}) and (\ref{eq 4}), which concern
the cases of (\ref{eq 5}) given by the countable list of diagrams
\begin{equation*}
  \xymatrix{
    K=\partial\Delta^{n} \ar[r] \ar[d] & \Delta^{n} \ar[d] \\
    L=(\partial\Delta^{n} \times \Delta^{1}) \cup (\Delta^{n} \times \{ 0 \})
    \ar[r] & \Delta^{n} \times \Delta^{1}
  }
\end{equation*}
where $n \geq 0$.  We also suppose that the regular cardinals $\lambda
> \alpha$ are chosen as above.

Then for a fixed diagram (\ref{eq 5}), or fixed $n \geq 0$, the list
of associated extension problems (\ref{eq 9}) is determined by the
size of the set of functors $NL \to I$.  The category $I$ is
$\lambda$-bounded, while there are countably many finite complexes $L
= \Delta^{n} \times \Delta^{1}$ of interest, and so the entire list of
relevant extension problems is $\lambda$-bounded.

\begin{lemma}\label{lem 11}
  Suppose that the regular cardinals $\lambda > \alpha$ are chosen as
  above.  Suppose that the monomorphism $i: X \to Y$ is a
  pro-equivalence, and that $A$ is a $\lambda$-bounded subobject of
  $Y$. Then there is a $\lambda$-bounded subobject $B \subset Y$ with
  $A \subset B$ such that the map $B \cap X \to B$ is a
  pro-equivalence.
  \end{lemma}

\begin{proof}
Consider the commutative diagram of inclusions of $I$-diagrams
\begin{equation*}
\xymatrix{
A \cap X \ar[r] \ar[d] & X \ar[d]^{i} \\
A \ar[r] & Y
}
\end{equation*}
where $A$ is a $\lambda$-bounded subobject of $Y$. Suppose that the
lifting problem (\ref{eq 9}) can be solved up to subdivision for $i: X
\to Y$, and that $Z$ is the fibrant model $\mathcal{L}(L_{K}Y
\cup_{L_{K}X} L_{L}X)$. Consider the diagram
\begin{equation*}
\xymatrix{
L_{K}A \cup_{L_{K}(A \cap X)} L_{L}(A \cap X) \ar[r] \ar[d]
& L_{K}Y \cup_{L_{K}X} L_{L}X \ar[r] \ar[d] & Z \\
L_{K'}A \cup_{L_{K'}(A \cap X)} L_{L'}(A \cap X) \ar[r]_-{j}
& L_{K'}Y \cup_{L_{K'}X} L_{L'}X \ar@{.>}[ur]_{\theta} \\
}
\end{equation*}

The image of the composite $\theta \cdot j$ is $\lambda$-bounded, and
$Z=\mathcal{L}(L_{K}Y \cup_{L_{K}X} L_{L}X)$ is a filtered colimit of
the subobjects $\mathcal{L}(L_{K}C \cup_{L_{K}(C \cap X)} L_{L}(C \cap
X))$, as $C$ varies through the $\lambda$-bounded subobjects of $Y$,
so the image of $\theta \cdot j$
factors through $\mathcal{L}(L_{K}A_{1} \cup_{L_{K}(A_{1} \cap X)}
L_{L}(A_{1} \cap X))$ for some $\lambda$-bounded $A_{1} \subset Y$ with
$A \subset A_{1}$. This can be done simultaneously for the full
$\lambda$-bounded list of extension problems.

Repeat this construction $\lambda$ times, and let $B = \varinjlim_{t <
  \lambda}\ A_{t}$. Then the map
\begin{equation*}
\varinjlim_{t < \lambda}\ \mathcal{L}(L_{K}A_{t} \cup_{L_{K}(A_{t} \cap X)}
L_{L}(A_{t} \cap X)) \to
\mathcal{L}(L_{K}B \cup_{L_{K}(B \cap X)}
L_{L}(B \cap X))
\end{equation*}
is an isomorphism, and there is a commutative diagram
\begin{equation*}
\xymatrix{
L_{K}B \cup_{L_{K}(B \cap X)} L_{L}(B \cap X) \ar[r]^-{j} \ar[d]
& \mathcal{L}(L_{K}B \cup_{L_{K}(B \cap X)} L_{L}(B \cap X)) \\
L_{K'}B \cup_{L_{K'}(B \cap X)} L_{L'}(B \cap X) \ar[ur]
}
\end{equation*}
where $j$ is the fibrant model map. This holds for all relevant
extension problems, so that the map $B \cap X \to B$ is a
pro-equivalence.
\end{proof}

Lemma \ref{lem 11} is the generalization of Lemma \ref{lem 1}, to the
case of arbitrary small index categories $I$. Proposition \ref{prop 5}
follows from Lemma \ref{lem 1}, via a sequence of formal steps given
by Lemma \ref{lem 2}, Corollary \ref{cor 3} and Lemma \ref{lem 4}.
These same results apply to the present case of arbitrary
$I$-diagrams, starting from Lemma \ref{lem 11}, giving the
following result:

\begin{theorem}\label{th xxx}
Suppose that $I$ is a small category $I$, and that $\mathcal{C}$
is a Grothendieck site. The category $s\Pre^{I}$ of $I$-diagrams in
simplicial presheaves on $\mathcal{C}$, together with the classes of
cofibrations, pro-equivalences and pro-fibrations, satisfies the
axioms for a left proper closed simplicial model category. This model
structure is cofibrantly generated.
\end{theorem}

\bibliographystyle{plain} 
\bibliography{spt}

\end{document}